\documentclass[oneside,reqno,english]{amsart}
\usepackage[T1]{fontenc}
\usepackage[latin9]{inputenc}
\usepackage{babel}
\usepackage{prettyref}
\usepackage{mathrsfs}
\usepackage{amstext}
\usepackage{amsthm}
\usepackage{amssymb}
\usepackage[pdfusetitle,
 bookmarks=true,bookmarksnumbered=false,bookmarksopen=false,
 breaklinks=false,pdfborder={0 0 0},pdfborderstyle={},backref=false,colorlinks=false]
 {hyperref}
\hypersetup{
 colorlinks=true,citecolor=blue,linkcolor=blue,linktocpage=true}

\makeatletter
\numberwithin{equation}{section}
\numberwithin{figure}{section}
\theoremstyle{plain}
\newtheorem{thm}{\protect\theoremname}[section]
\theoremstyle{remark}
\newtheorem{rem}[thm]{\protect\remarkname}
\theoremstyle{definition}
\newtheorem{defn}[thm]{\protect\definitionname}
\theoremstyle{plain}
\newtheorem{prop}[thm]{\protect\propositionname}
\theoremstyle{plain}
\newtheorem{cor}[thm]{\protect\corollaryname}

\@ifundefined{date}{}{\date{}}
\newrefformat{cor}{Corollary~\ref{#1}}
\newrefformat{subsec}{Section~\ref{#1}}
\newrefformat{lem}{Lemma~\ref{#1}}
\newrefformat{thm}{Theorem~\ref{#1}}
\newrefformat{sec}{Section~\ref{#1}}
\newrefformat{chap}{Chapter~\ref{#1}}
\newrefformat{prop}{Proposition~\ref{#1}}
\newrefformat{exa}{Example~\ref{#1}}
\newrefformat{tab}{Table~\ref{#1}}
\newrefformat{rem}{Remark~\ref{#1}}
\newrefformat{def}{Definition~\ref{#1}}
\newrefformat{fig}{Figure~\ref{#1}}
\newrefformat{claim}{Claim~\ref{#1}}

\makeatother

\providecommand{\corollaryname}{Corollary}
\providecommand{\definitionname}{Definition}
\providecommand{\propositionname}{Proposition}
\providecommand{\remarkname}{Remark}
\providecommand{\theoremname}{Theorem}

\begin{document}
\title[]{Operator-Valued Kernels, Machine Learning, and Dynamical Systems}
\author{Palle E.T. Jorgensen}
\address{(Palle E.T. Jorgensen) Department of Mathematics, The University of
Iowa, Iowa City, IA 52242-1419, U.S.A.}
\email{palle-jorgensen@uiowa.edu}
\author{James Tian}
\address{(James F. Tian) Mathematical Reviews, 416 4th Street Ann Arbor, MI
48103-4816, U.S.A.}
\email{jft@ams.org}
\begin{abstract}
In the context of kernel optimization, we prove a result that yields
new factorizations and realizations. Our initial context is that of
general positive operator-valued kernels. We further present implications
for Hilbert space-valued Gaussian processes, as they arise in applications
to dynamics and to machine learning. Further applications are given
in non-commutative probability theory, including a new non-commutative
Radon--Nikodym theorem.
\end{abstract}


\maketitle
\tableofcontents{}

\section{\protect\label{sec:1}Introduction}

Machine learning, particularly kernel methods, heavily relies on the
concept of positive definite (p.d.) kernels. These kernels allow the
transformation of data into higher-dimensional spaces where linear
separability is possible, leading to effective classification and
regression models. The positive operator-valued kernels discussed
in our paper extend this idea to more complex structures, such as
Hilbert space-valued Gaussian processes. These processes enable the
construction of probabilistic models that can handle uncertainty and
provide predictive distributions.

For instance, in support vector machines (SVMs) and Gaussian process
regression (GPR), the choice and optimization of kernels are crucial
for model performance. Our factorization results provide new ways
to construct and optimize kernels, potentially leading to more efficient
algorithms and improved generalization in machine learning models.
This is mainly due to the induced scalar-valued counterparts of these
operator-valued kernels. More specifically, for a fixed $\mathcal{L}\left(H\right)$-valued
kernel $K\left(s,t\right)$, $\left(s,t\right)\in S\times S$, there
is a family of scalar-valued kernels $K_{\rho}\left(s,t\right)=\text{Trace}\left(\rho K\left(s,t\right)\right)$,
indexed by positive trace class operators $\rho$ in the Hilbert space
$H$. This, in turn, yields the following kernel optimization model:
\begin{equation}
\min_{f,\rho}\left\{ \sum_{i}\left|f\left(s_{i}\right)-c_{i}\right|^{2}+\beta\left\Vert f\right\Vert _{H_{K_{\rho}}}:f\in H_{K_{\rho}},\rho>0\right\} \label{eq:a1}
\end{equation}
in which $H_{K_{\rho}}$ is the reproducing Hilbert space (RKHS) of
$K_{\rho}$. 

In a Bayesian interpretation of \eqref{eq:a1}, we impose a prior
belief on the structure of the kernel function through the constraint
$\rho$. This positive operator $\rho$ can be seen as defining a
distribution over the ambient space of kernel functions (assuming
$\text{Trace}\left(\rho\right)=1$), reflecting prior knowledge or
assumptions about the functions we aim to model.

The data fit term in \eqref{eq:a1} plays the role of the likelihood
function. It assesses how well a particular kernel function, determined
by $\rho$, explains the observed data. Essentially, this term measures
the likelihood of the data given the chosen kernel. By combining the
prior and likelihood, we can derive a posterior distribution over
the space of kernel functions. This posterior represents our updated
belief about the kernel after observing the data. The solution to
our optimization problem \eqref{eq:a1} can then be interpreted as
a Maximum a Posteriori (MAP) estimate, meaning that a solution corresponds
to the kernel function that is most probable under the posterior distribution,
given the observed data and our prior assumptions.

\textbf{Main results.} In this paper, motivated by the model \eqref{eq:a1},
we present a result that offers a canonical link between a general
setting for operator-valued completely positive maps and their induced
scalar-valued counterparts. We further provide applications to a variety
of neighboring areas, including the following six closely interrelated
topics: (i) operator-valued Gaussian processes and their associated
covariance structures; (ii) universal factorizations; (iii) non-commutative
operator-valued Radon-Nikodym derivatives and their applications to
quantum gates and quantum states; (iv) partial orders on operator-valued
completely positive maps via their associated reproducing kernel Hilbert
spaces; (v) intertwining operators for representations induced from
completely positive maps; and (vi) applications of (iii) to completely
positive maps and associated quantum gates.

\textbf{Organization.} In \prettyref{sec:2} we present the general
framework, and the main theorems, for operator-valued completely positive
maps, as well as their associated structures, including operator-valued
Gaussian processes. The applications to non-commutative operator-valued
Radon-Nikodym derivatives are discussed in \prettyref{sec:3}, while
the focus in \prettyref{sec:4} is that of completely positive maps.

\textbf{Literature.} For the reader\textquoteright s benefit, we include
the following citations across various relevant topics: For the theory
of positive definite kernels \cite{MR4439542,MR4414884,MR4201338},
Hilbert space-valued Gaussian processes \cite{MR4509085,MR3959187,MR3699787},
completely positive maps \cite{MR4654017,MR4526376,MR4287855}, quantum
gates \cite{MR4740581,MR4737207,MR4733151,MR4703397}, and operator-valued
Radon-Nikodym derivatives \cite{MR4137283}.

The literature on the theory of reproducing kernels is vast, encompassing
both theoretical foundations and recent applications. For a comprehensive
overview, including the latest developments, we refer to the following
resources: \cite{MR80878,MR31663,MR4302453,MR3526117,MR1305949,MR4690276,MR3700848,MR1126127,MR1200633,MR3402823}.

We also recommend two pioneering books for the general theory of dilations
of irreversible evolutions in algebraic quantum theory: \cite{MR489494,MR582649}.
These books cover both the theoretical aspects and applications of
completely positive maps, including semigroups.

Lastly, for insights into machine learning, particularly the kernel
approach and its dynamical perspective, we refer to \cite{MR4736516,MR4690614,MR4675283,MR4676260,MR4648377,MR4586825,MR4291375}
and the papers cited therein. 

\textbf{Notation.} Throughout the paper, we use the physics convention
that inner products are linear in the second variable. Let $\left|a\left\rangle \right\langle b\right|$
denote Dirac's rank-1 operator, $c\mapsto a\left\langle b,c\right\rangle $.
$\mathcal{L}\left(H\right)$ denotes the algebra of all bounded linear
operators on a Hilbert space $H$. 

For a positive definite (p.d.) kernel $K:S\times S\rightarrow\mathbb{C}$,
let $H_{K}$ be the corresponding reproducing kernel Hilbert space
(RKHS). $H_{K}$ is the Hilbert completion of 
\begin{equation}
span_{\mathbb{C}}\left\{ K_{y}\left(\cdot\right):=K\left(\cdot,y\right)\mid y\in S\right\} \label{eq:A1}
\end{equation}
with respect to the inner product 
\begin{equation}
\left\langle \sum_{i}c_{i}K\left(\cdot,x_{i}\right),\sum_{i}d_{j}K\left(\cdot,y_{j}\right)\right\rangle _{H_{K}}:=\sum_{i}\sum_{j}\overline{c}_{i}d_{j}K\left(x_{i},x_{j}\right).\label{eq:A2}
\end{equation}
The following reproducing property holds: 
\begin{equation}
\varphi\left(x\right)=\left\langle K\left(\cdot,x\right),\varphi\right\rangle _{H_{K}},\forall x\in S,\:\forall\varphi\in H_{K}.\label{eq:A3}
\end{equation}

Any scalar-valued kernel $K$ as above is associated with a zero-mean
Gaussian process $W_{s}\sim N\left(0,K\left(s,s\right)\right)$, where
$K$ is the covariance: 
\begin{equation}
K\left(s,t\right)=\mathbb{E}\left[\overline{W_{s}}W_{t}\right].\label{eq:A4}
\end{equation}

An $\mathcal{L}\left(H\right)$-valued kernel $K:S\times S\rightarrow\mathcal{L}\left(H\right)$
is p.d. if 
\begin{equation}
\sum_{i=1}^{n}\left\langle a_{i},K\left(s_{i},s_{j}\right)a_{j}\right\rangle _{H}\geq0\label{eq:A5}
\end{equation}
for all $\left(a_{i}\right)_{1}^{n}$ in $H$, and all $n\in\mathbb{N}$.

\section{\protect\label{sec:2}$\mathcal{L}\left(H\right)$-valued kernels
and $H$-valued Gaussian processes}

In passing from kernels to kernel Hilbert spaces, there are two approaches:
one is abstract, and the other is constructive. The first approach
is an abstract Hilbert completion. Despite its drawbacks, this has
been the method used in the literature until now. The second approach
is constructive. It is new, and we present it below. As we show in
the subsequent sections of our paper, the constructive design is well-suited
to a variety of applications in machine learning and dynamics.

We begin by revisiting a universal construction for operator-valued
positive definite kernels. This construction will be adapted for various
applications, which will be explored in the following sections.
\begin{thm}[Universal Factorization]
\label{thm:b2} Let $K:S\times S\rightarrow\mathcal{L}\left(H\right)$
be a p.d. kernel. Then there exists a RKHS $H_{\tilde{K}}$, and a
family of operators $V\left(s\right):H\rightarrow H_{\tilde{K}}$,
$s\in S$, such that 
\begin{equation}
H_{\tilde{K}}=\overline{span}\left\{ V\left(s\right)a:a\in H,s\in S\right\} 
\end{equation}
and 
\begin{equation}
K\left(s,t\right)=V\left(s\right)^{*}V\left(t\right).\label{eq:b2}
\end{equation}

Conversely, if there is a Hilbert space $L$ and operators $V\left(s\right):H\rightarrow L$,
$s\in S$, such that 
\begin{equation}
L=\overline{span}\left\{ V\left(s\right)a:a\in H,s\in S\right\} \label{eq:b3}
\end{equation}
and \eqref{eq:b2} holds, then $L\simeq H_{\tilde{K}}$. 
\end{thm}

\begin{proof}
For a detailed proof, we refer the reader to \cite{JT2024canonical}.
We provide a brief outline of the main steps below: 

Given a p.d. kernel $K:S\times S\rightarrow\mathcal{L}\left(H\right)$,
define $\tilde{K}:\left(S\times H\right)\times\left(S\times H\right)\rightarrow\mathbb{C}$
as 
\begin{equation}
\tilde{K}\left(\left(s,a\right),\left(t,b\right)\right):=\left\langle a,K\left(s,t\right)b\right\rangle _{H}.\label{eq:b4}
\end{equation}
Then $\tilde{K}$ is a scalar-valued p.d. kernel. Let $H_{\tilde{K}}$
be the associated RKHS.

Set $V\left(s\right):H\rightarrow H_{\tilde{K}}$ by 
\begin{equation}
V\left(s\right)a=\tilde{K}_{\left(s,a\right)}=\tilde{K}\left(\cdot,\left(s,a\right)\right):\left(S\times H\right)\rightarrow\mathbb{C},\quad\forall a\in H.\label{eq:b5}
\end{equation}
One verifies that 
\[
V\left(s\right)^{*}\tilde{K}\left(\cdot,\left(t,b\right)\right)=K\left(s,t\right)b.
\]
and the factorization $K\left(s,t\right)=V\left(s\right)^{*}V\left(t\right)$
holds.
\end{proof}
\begin{rem}
The setting in \prettyref{thm:b2} includes and extends well known
constructions in classical dilation theory. Notably, transitioning
from an operator-valued kernel $K$ to the scalar-valued positive
definite kernel $\tilde{K}$ enables a function-based approach to
dilation theory, contrasting with the traditional abstract spaces
of equivalence classes. The literature is vast, and here we call attention
to \cite{MR2743416}, and the papers cited there.
\end{rem}

Building upon Ito\textquoteright s seminal work, Hilbert spaces and
their corresponding operators have become essential tools in the stochastic
analysis of Gaussian processes. For more detailed insights, see e.g.,
\cite{MR3402823,MR4302453}, as well as the earlier works cited therein.
We explore two principal types of Gaussian processes in this field:
(i) scalar-valued processes that are indexed by a Hilbert space, typically
through an Ito-isometry; and (ii) processes where the Gaussian values
are embedded directly within a Hilbert space (see e.g., \cite{MR4641110,MR4414825,MR4101087,MR4073554,MR3940383,doi:10.1142/S0219025723500200}).
This paper concentrates on the second type, which provides enhanced
adaptability in modeling covariance structures---a key element in
processing large datasets. We now proceed to discuss the details of
this approach.
\begin{thm}
Every operator-valued p.d. kernel $K:S\times S\rightarrow\mathcal{L}\left(H\right)$
is associated with an $H$-valued Gaussian process $\left\{ W\left(s\right)\right\} _{s\in S}$
with zero-mean, realized on some probability space $\left(\Omega,\mathbb{P}\right)$,
such that 
\begin{equation}
K\left(s,t\right)=\int_{\Omega}\left|W\left(s\right)\left\rangle \right\langle W\left(t\right)\right|d\mathbb{P}.\label{eq:b6}
\end{equation}
Conversely, every $H$-valued Gaussian process is obtained from such
a $\mathcal{L}\left(H\right)$-valued kernel. 
\end{thm}

\begin{rem}
More precisely, the identity \eqref{eq:b6} holds in the sense that
\begin{equation}
\mathbb{E}\left[\left\langle a,W\left(s\right)\right\rangle _{H}\left\langle W\left(t\right),b\right\rangle _{H}\right]=\left\langle a,K\left(s,t\right)b\right\rangle _{H}.
\end{equation}
for all $s,t\in S$, and all $a,b\in H$. For a detailed proof, see
\cite{JT2024canonical}. 
\end{rem}

\begin{defn}
Consider the $H$-valued Gaussian process $W_{s}:\Omega\rightarrow H$,
\begin{equation}
W\left(t\right)=\sum_{i}\left(V\left(t\right)^{*}\varphi_{i}\right)Z_{i},\label{eq:d8}
\end{equation}
where $\left(\varphi_{i}\right)$ is an orthonormal basis (ONB) in
$H_{\tilde{K}}$. Following standard conventions, here $\left\{ Z_{i}\right\} $
refers to a choice of an independent identically distributed (i.i.d.)
system of standard scalar Gaussian random variables, i.e., $Z_{i}\sim N(0,1)$,
and with an index matching the choice of ONB.
\end{defn}

\begin{thm}
\label{thm:c3}We have 
\begin{equation}
\mathbb{E}\left[\left\langle a,W\left(s\right)\right\rangle _{H}\left\langle W\left(t\right),b\right\rangle _{H}\right]=\left\langle a,K\left(s,t\right)b\right\rangle _{H}.\label{eq:d9}
\end{equation}
\end{thm}

\begin{proof}
Note that $\mathbb{E}\left[Z_{i}Z_{j}\right]=\delta_{i,j}$. From
this, we get
\begin{align*}
\text{LHS}_{\left(\ref{eq:d9}\right)} & =\mathbb{E}\left[\left\langle a,W\left(s\right)\right\rangle _{H}\left\langle W\left(t\right),b\right\rangle _{H}\right]\\
 & =\sum_{i}\sum_{j}\left\langle a,V\left(s\right)^{*}\varphi_{i}\right\rangle \left\langle V\left(t\right)^{*}\varphi_{j},b\right\rangle \mathbb{E}\left[Z_{i}Z_{j}\right]\\
 & =\sum_{i}\left\langle V\left(s\right)a,\varphi_{i}\right\rangle _{H_{\tilde{K}}}\left\langle \varphi_{i},V\left(t\right)b\right\rangle _{H_{\tilde{K}}}\\
 & =\left\langle V\left(s\right)a,V\left(t\right)b\right\rangle _{H_{\tilde{K}}}=\left\langle a,V\left(s\right)^{*}V\left(t\right)b\right\rangle _{H}=\left\langle a,K\left(s,t\right)b\right\rangle _{H}.
\end{align*}
\end{proof}
\begin{rem}
The Gaussian process $\left\{ W\left(t\right)\right\} _{t\in S}$
in \eqref{eq:d8} is well defined and possesses the stated properties.
This is an application of the central limit theorem to the choice
$\left\{ Z_{i}\right\} $ of i.i.d. $N\left(0,1\right)$ Gaussians
on the right-side of \eqref{eq:d8}.

Varying the choices of ONBs $\left(\varphi_{i}\right)$ and i.i.d.
$N\left(0,1\right)$ Gaussian random variables $\left(Z_{i}\right)$
will result in different Gaussian processes $\left\{ W\left(t\right)\right\} _{t\in S}$,
but all will adhere to the covariance condition specified in \eqref{eq:d9}.
Importantly, once the ONB $\left(\varphi_{i}\right)$ and the random
variables $\left(Z_{i}\right)$ are fixed, the resulting Gaussian
process $\left\{ W\left(t\right)\right\} _{t\in S}$ is uniquely determined
by its first two moments (mean and covariance).
\end{rem}

\section{\protect\label{sec:3}Dynamics and a non-commutative Radon--Nikodym
theorem}

As is known, the theory of von Neumann algebras offers a framework
for non-commutative measure theory, see e.g., \cite{MR174992}. In
this interpretation, the projections in the algebra are the characteristic
functions of the (non-commuting) \textquotedblleft measurable sets.\textquotedblright{}
From the elements of the Hilbert space $H$ we then build the bounded
measurable functions; and the (normal) states are the probability
measures on the underlying (non-commutative) measure space. Here we
shall supplement this theory with Aronszajn\textquoteright s notion
of systems of ordered p.d. kernels and contractive containment of
the corresponding reproducing kernel Hilbert spaces (RKHSs.) From
this we then build a natural (but different) notion of non-commutative
Radon--Nikodym derivatives. 

For the sake of completeness, we include a proof of Aronszajn\textquoteright s
inclusion theorem below. It states that, for two (scalar-valued) p.d.
kernels $K$ and $L$ on $S\times S$, $K\leq L$ (i.e., $L-K$ is
p.d.) if and only if $H_{K}$ is contractively contained in $H_{L}$
(see e.g., \cite{MR51437}). 
\begin{thm}[Aronszajn]
\label{thm:21}If $K\leq L$, then $H_{K}\subset H_{L}$ and $\left\Vert f\right\Vert _{H_{L}}\leq\left\Vert f\right\Vert _{H_{K}}$
for all $f\in H_{K}$. 
\end{thm}

\begin{proof}
Let $f\in H_{K}$. To verify that $f\in H_{L}$, it is equivalent
to show that 
\[
\left|\sum c_{i}f\left(s_{i}\right)\right|^{2}\leq\text{const}\cdot\left\Vert \sum c_{i}L_{s_{i}}\right\Vert _{H_{L}}^{2}
\]
for all $\left(c_{i}\right)_{i=1}^{n}$ in $\mathbb{C}$, $\left(s_{i}\right)_{i=1}^{n}$
in $S$, and $n\in\mathbb{N}$. This follows from the reproducing
property, and the assumption $K\leq L$: 
\begin{align*}
\left|\sum c_{i}f\left(s_{i}\right)\right|^{2} & =\left|\left\langle \sum\overline{c_{i}}K_{s_{i}},f\right\rangle _{H_{K}}\right|^{2}\leq\left\Vert f\right\Vert _{H_{K}}^{2}\left\Vert \sum\overline{c_{i}}K_{s_{i}}\right\Vert _{H_{K}}^{2}\\
 & \leq\left\Vert f\right\Vert _{H_{K}}^{2}\sum\overline{c_{i}}c_{j}K\left(s_{i},s_{j}\right)\leq\left\Vert f\right\Vert _{H_{K}}^{2}\sum\overline{c_{i}}c_{j}L\left(s_{i},s_{j}\right)\\
 & =\left\Vert f\right\Vert _{H_{K}}^{2}\left\Vert \sum c_{i}L_{s_{i}}\right\Vert _{H_{L}}^{2}.
\end{align*}
\end{proof}
\begin{prop}
\label{prop:f1}Suppose $K,L$ are p.d. on $S\times S$, and $K\leq L$.
There exists a unique positive selfadjoint operator $T$ on $H_{L}$,
such that $0\leq T\leq I_{H_{L}}$, and 

\begin{equation}
K\left(s,t\right)=\left\langle T^{1/2}L_{s},T^{1/2}L_{t}\right\rangle _{H_{L}},\quad s,t\in S.\label{eq:t4-1}
\end{equation}
\end{prop}

\begin{proof}
Define a map $\Phi$ as 
\[
\Phi\left(L_{s},L_{t}\right)=K\left(s,t\right)
\]
and extend it by linearity: 
\[
\Phi\left(\sum_{i=1}^{m}c_{i}L_{s_{i}},\sum_{j=1}^{n}d_{j}L_{t_{j}}\right)=\sum_{i=1}^{m}\sum_{j=1}^{n}\overline{c_{i}}d_{j}K\left(s_{i},t_{j}\right).
\]
Then $\Phi$ extends by density to a unique bounded sesquilinear form
on $H_{L}$. 

The remaining assertions follow from the general theory of quadratic
forms. In particular, if $j:H_{K}\rightarrow H_{L}$ is the inclusion
map, then $T=jj^{*}:H_{L}\rightarrow H_{L}$ is the desired operator. 
\end{proof}
We now consider $\mathcal{L}\left(H\right)$-valued p.d. kernels. 
\begin{defn}
\label{def:3-3}Suppose $K,L$ are $\mathcal{L}\left(H\right)$-valued
p.d. kernels defined on $S\times S$. We say that $K\leq L$ if 
\begin{equation}
\sum\left\langle a_{i},K\left(s_{i},s_{j}\right)a_{j}\right\rangle _{H}\leq\sum\left\langle a_{i},L\left(s_{i},s_{j}\right)a_{j}\right\rangle _{H}\label{eq:c4}
\end{equation}
for all $\left(s_{i}\right)_{i=1}^{n}$ in $S$, $\left(a_{i}\right)_{i=1}^{n}$
in $H$, and $n\in\mathbb{N}$.
\end{defn}

Let $\tilde{K}$, $\tilde{L}$ be the associated scalar-valued p.d.
kernels on $S\times S$, and $H_{\tilde{K}},H_{\tilde{L}}$ be the
respective RKHSs (see \eqref{eq:b4}). \prettyref{def:3-3} means
that
\begin{equation}
K\leq L\Longleftrightarrow\tilde{K}\leq\tilde{L}.\label{eq:c5}
\end{equation}

\begin{thm}
\label{thm:c4}Let $K,L:S\times S\rightarrow\mathcal{L}\left(H\right)$
be operator-valued p.d. kernels on $S\times S$. Let $\tilde{K}$,
$\tilde{L}$ be the associated scalar-valued p.d. kernels, and $H_{\tilde{K}}$,
$H_{\tilde{L}}$ be the respective RKHSs. 

The following are equivalent: 
\begin{enumerate}
\item $K\leq L$ in the sense of \eqref{eq:c4}--\eqref{eq:c5}.
\item There exists a positive selfadjoint operator $T:H_{\tilde{L}}\rightarrow H_{\tilde{L}}$,
$0\leq T\leq I_{H_{\tilde{L}}}$, such that 
\begin{equation}
K\left(s,t\right)=V_{L}\left(s\right)^{*}TV_{L}\left(t\right).\label{eq:C4}
\end{equation}
Equivalently, 
\begin{eqnarray}
V_{K}\left(t\right) & = & T^{1/2}V_{L}\left(t\right).\label{eq:C5}
\end{eqnarray}
Here, $L\left(s,t\right)=V_{L}\left(s\right)^{*}V_{L}\left(t\right)$
and $K\left(s,t\right)=V_{K}\left(s\right)^{*}V_{K}\left(t\right)$
are the respective canonical factorizations of $L$ and $K$ (\prettyref{thm:b2}). 
\end{enumerate}
\end{thm}

\begin{proof}
(1) $\Rightarrow$ (2) Assume $K\le L$. By \prettyref{thm:21}, $H_{\tilde{K}}\subset H_{\tilde{L}}$
and $\left\Vert f\right\Vert _{H_{\tilde{L}}}\leq\left\Vert f\right\Vert _{H_{\tilde{K}}}$
for all $f\in H_{\tilde{K}}$. 

Recall that $L$ factors as 
\[
L\left(s,t\right)=V_{L}\left(s\right)^{*}V_{L}\left(t\right)
\]
where $V_{L}\left(t\right):H\rightarrow H_{\tilde{L}}$ is given by
$V_{L}\left(t\right)a=\tilde{L}_{\left(t,a\right)}$, for all $t\in S$
and $a\in H$ (see \eqref{eq:b5}).

Apply \prettyref{prop:f1} to the scalar-valued kernels $\tilde{K}$
and $\tilde{L}$: There exists a unique selfadjoint operator $T$
in $H_{\tilde{L}}$, with $0\leq T\leq I_{H_{\tilde{L}}}$, so that
\begin{align*}
\left\langle a,K\left(s,t\right)b\right\rangle _{H} & =\left\langle \tilde{K}_{\left(s,a\right)},\tilde{K}_{\left(t,b\right)}\right\rangle _{H_{\tilde{K}}}\\
 & =\left\langle T^{1/2}\tilde{L}_{\left(s,a\right)},T^{1/2}\tilde{L}_{\left(t,b\right)}\right\rangle _{H_{\tilde{L}}}\\
 & =\left\langle T^{1/2}V_{L}\left(s\right)a,T^{1/2}V_{L}\left(t\right)b\right\rangle _{H_{\tilde{L}}}\\
 & =\left\langle a,V_{L}\left(s\right)^{*}TV_{L}\left(t\right)b\right\rangle _{H}
\end{align*}
for all $a,b\in H$, and all $s,t\in S$. Therefore, \eqref{eq:C4}
holds. The identity \eqref{eq:C5} follows from the above argument
and the proof of \prettyref{prop:f1}. 

(2) $\Rightarrow$ (1) Conversely, from \eqref{eq:C4} and the fact
that $0\leq T\leq I_{H_{\tilde{L}}}$, we have 
\begin{align*}
\sum_{i,j}\left\langle a_{i},K\left(s_{i},s_{j}\right)a_{j}\right\rangle _{H} & =\sum_{i,j}\left\langle a_{i},V_{L}\left(s_{i}\right)^{*}TV_{L}\left(s_{j}\right)a_{j}\right\rangle _{H}\\
 & =\left\Vert T^{1/2}\sum_{i}V_{L}\left(s_{i}\right)a_{i}\right\Vert _{H_{\tilde{L}}}^{2}\leq\left\Vert \sum_{i}V_{L}\left(s_{i}\right)a_{i}\right\Vert _{H_{\tilde{L}}}^{2}\\
 & =\sum_{i,j}\left\langle a_{i},V_{L}\left(s_{i}\right)^{*}V_{L}\left(s_{j}\right)a_{j}\right\rangle _{H}\\
 & =\sum_{i,j}\left\langle a_{i},L\left(s_{i},s_{j}\right)a_{j}\right\rangle _{H}
\end{align*}
and so $K\leq L$. 
\end{proof}
\begin{cor}
Suppose $K,L:S\times S\rightarrow\mathcal{L}\left(H\right)$ p.d.,
and $K\leq L$. Let $T=dK/dL$ be the Radon-Nikodym derivative from
\prettyref{thm:c4}. 

Let $W_{L}\left(t\right)$ be the $H$-valued Gaussian process from
\eqref{eq:d8}, i.e., 
\begin{equation}
W_{L}\left(t\right)=\sum_{i}\left(V_{L}\left(t\right)^{*}\varphi_{i}\right)Z_{i}
\end{equation}
Set 
\begin{equation}
W_{K}\left(t\right):=\sum_{i}\left(V_{L}\left(t\right)^{*}T^{1/2}\varphi_{i}\right)Z_{i}.
\end{equation}
Then $K$ admits the following decomposition 
\begin{equation}
K\left(s,t\right)=\int_{\Omega}\left|W_{K}\left(s\right)\left\rangle \right\langle W_{K}\left(t\right)\right|d\mathbb{P}.\label{eq:C8}
\end{equation}
\end{cor}

\begin{proof}
Recall that \eqref{eq:C8} is equivalent to 
\begin{equation}
\mathbb{E}\left[\left\langle a,W_{K}\left(s\right)\right\rangle _{H}\left\langle W_{K}\left(t\right),b\right\rangle _{H}\right]=\left\langle a,K\left(s,t\right)b\right\rangle _{H},\label{eq:c8}
\end{equation}
for all $a,b\in H$ and $s,t\in S$.

Given $L:S\times S\rightarrow\mathcal{L}\left(H\right)$ p.d., recall
that 
\[
L\left(s,t\right)=V_{L}\left(s\right)^{*}V_{L}\left(t\right)
\]
where $V_{L}\left(t\right):H\rightarrow H_{\tilde{L}}=$ the RKHS
of $\tilde{L}$. 

Let $\left(\varphi_{i}\right)$ be an ONB in $H_{\tilde{L}}$, and
apply the identity 
\[
I_{H_{\tilde{L}}}=\sum_{i}\left|\varphi_{i}\left\rangle \right\langle \varphi_{i}\right|
\]
we then get 
\begin{align*}
L\left(s,t\right) & =V_{L}\left(s\right)^{*}I_{H_{\tilde{L}}}V_{L}\left(t\right)\\
 & =\sum_{i}\left|V_{L}\left(s\right)^{*}\varphi_{i}\left\rangle \right\langle V_{L}\left(t\right)^{*}\varphi_{i}\right|.
\end{align*}
Similarly, using $V_{K}\left(t\right)=T^{1/2}V_{L}\left(t\right)$
from \eqref{eq:C5}, 
\begin{equation}
K\left(s,t\right)=\sum_{i}\left|V_{L}\left(s\right)^{*}T^{1/2}\varphi_{i}\left\rangle \right\langle V_{L}\left(t\right)^{*}T^{1/2}\varphi_{i}\right|.\label{eq:c9}
\end{equation}
Thus, 
\begin{eqnarray*}
\text{l.h.s.}_{\left(\ref{eq:c8}\right)} & = & \sum_{i,j}\left\langle a,V_{L}\left(s\right)^{*}T^{1/2}\varphi_{i}\right\rangle _{H}\left\langle V_{L}\left(t\right)^{*}T^{1/2}\varphi_{j},b\right\rangle _{H}\underset{=\delta_{i,j}}{\underbrace{\mathbb{E}\left[Z_{i}Z_{j}\right]}}\\
 & = & \sum_{i}\left\langle a,V_{L}\left(s\right)^{*}T^{1/2}\varphi_{i}\right\rangle _{H}\left\langle V_{L}\left(t\right)^{*}T^{1/2}\varphi_{i},b\right\rangle _{H}\\
 & = & \left\langle a,\left(\sum_{i}\left|V_{L}\left(s\right)^{*}T^{1/2}\varphi_{i}\left\rangle \right\langle V_{L}\left(t\right)^{*}T^{1/2}\varphi_{i}\right|\right)b\right\rangle _{H}\\
 & \underset{\left(\ref{eq:c9}\right)}{=} & \left\langle a,K\left(s,t\right)b\right\rangle _{H}.
\end{eqnarray*}
\end{proof}

\section{\protect\label{sec:4}Completely positive maps, quantum channels}

The role of completely positive maps in physics can be described as
follows: they represent the type of transformation that occurs, for
example, when a beam in a certain mixed state passes through a device,
resulting in another beam in a different mixed state, thus accounting
for dissipative effects. These transformations must map states into
states, meaning they must be positive. However, complete positivity
is a stronger requirement. Under the standard assumption of unitary
dynamics for the overall system, completely positive maps emerge as
restrictions of representations in a larger system, such as the combined
system of the beam and the device. In mathematical terms, this is
captured by Stinespring's theorem (see \eqref{eq:d1}), which states
that any completely positive map can be realized as a compression
of unitary dynamics, and thus, is experimentally realizable.

Let $\mathfrak{A}$ and $\mathfrak{B}$ be unital $C^{*}$-algebras.
A map $\psi:\mathfrak{A}\rightarrow\mathfrak{B}$ is said to be completely
positive (CP) if $\psi\otimes I_{n}:\mathfrak{A}\otimes M_{n}\rightarrow\mathfrak{B}\otimes M_{n}$
is positive for all $n\in\mathbb{N}$. 

Let $\mathfrak{A}$ be a unital $C^{*}$-algebra. Suppose $\psi:\mathfrak{A}\rightarrow\mathcal{L}\left(H\right)$
is completely positive and $\psi\left(I\right)=I_{H}$. Stinespring's
dilation theorem states that, there exists a Hilbert space $\mathscr{K}$,
a representation $\pi:\mathfrak{A}\rightarrow\mathcal{L}\left(\mathscr{K}\right)$,
and an isometric embedding $V:H\rightarrow\mathscr{K}$ such that
\begin{equation}
\psi\left(A\right)=V^{*}\pi\left(A\right)V.\label{eq:d1}
\end{equation}
Further, $\left(\pi,V,\mathscr{K}\right)$ may be chosen to be minimal,
i.e., $\mathscr{K}=\overline{\pi\left(\mathfrak{A}\right)VH}$. In
that case, the dilation is unique up to unitary equivalence. 

We sketch a proof of \eqref{eq:d1} as an application of results in
Sections \ref{sec:2}--\ref{sec:3}. 
\begin{cor}
\label{cor:d1}Let $\psi:\mathfrak{A}\rightarrow\mathcal{L}\left(H\right)$
be as above, i.e., completely positive, and $\psi\left(I\right)=I$.
Let $K=K_{\varphi}:\mathfrak{A}\times\mathfrak{A}\rightarrow\mathcal{L}\left(H\right)$
be given by 
\[
K\left(A,B\right):=\psi\left(A^{*}B\right),\quad A,B\in\mathfrak{A},
\]
as a $\mathcal{L}\left(H\right)$-valued p.d. kernel on $\mathfrak{A}\times\mathfrak{A}$.
Let $\mathscr{K}=H_{\tilde{K}}=$ the RKHS of the associated scalar-valued
kernel $\tilde{K}$. Define $V=V\left(I\right):H\rightarrow\mathscr{K}$
by 
\[
Vh=\tilde{K}_{\left(I,h\right)}:\mathfrak{A}\times H\rightarrow\mathbb{C};
\]
set $\pi:\mathfrak{A}\rightarrow\mathcal{L}\left(\mathscr{K}\right)$
by 
\[
\pi\left(A\right)\tilde{K}_{\left(B,h\right)}=\tilde{K}_{\left(AB,h\right)}
\]
for all $\left(B,h\right)\in\mathfrak{A}\times H$. Then \eqref{eq:d1}
holds, and $\left(\pi,V,\mathscr{K}\right)$ is minimal. 
\end{cor}

\begin{proof}
This is a direct ``translation'' of \prettyref{thm:b2} to the setting
of CP maps. 

Specifically, when $K\left(A,B\right)=\psi\left(A^{*}B\right)$, it
factors into 
\[
K\left(A,B\right)=V^{*}\left(A\right)V\left(B\right)
\]
as in \eqref{eq:b2}. That is, 
\[
K\left(A,B\right)h=V^{*}\left(A\right)V\left(B\right)h=V^{*}\left(A\right)\tilde{K}_{\left(B,h\right)}=\psi\left(A^{*}B\right)h.
\]
By setting $A=I$ and $V:=V\left(I\right)$, this reduces to 
\begin{align*}
K\left(I,B\right)h & =V^{*}V\left(B\right)h=V^{*}\tilde{K}_{\left(B,h\right)}\\
 & =V^{*}\pi\left(B\right)\tilde{K}_{\left(I,h\right)}\\
 & =V^{*}\pi\left(B\right)Vh=\psi\left(B\right)h,
\end{align*}
which is \eqref{eq:d1}. 

Moreover, since 
\begin{align*}
\mathscr{K}=H_{\tilde{K}} & =\overline{span}\left\{ \tilde{K}_{\left(A,h\right)}:A\in\mathfrak{A},h\in H\right\} \\
 & =\overline{span}\left\{ \pi\left(A\right)Vh:A\in\mathfrak{A},h\in H\right\} =\overline{\pi\left(\mathfrak{A}\right)VH},
\end{align*}
the dilation is minimal. The assertion of uniqueness (up to unitary
equivalence) is immediate. 
\end{proof}
In the above setting, there is also a Radon--Nikodym type theorem
that characterizes all CP maps $\varphi,\psi$ for which $\varphi\leq\psi$,
in the sense that $\psi-\varphi$ is CP. In terms of the operator-valued
kernels, say $K_{\varphi}$ and $K_{\psi}$, we have $\varphi\leq\psi$
$\Longleftrightarrow$ $K_{\varphi}\leq K_{\psi}$, and the latter
is in the sense of \prettyref{def:3-3}.

We sketch below that this result can be derived as an application
of \prettyref{thm:c4}. 
\begin{cor}
\label{cor:d2}Let $\varphi,\psi:\mathfrak{A}\rightarrow\mathcal{L}\left(H\right)$
be CP maps. Let $\left(\pi_{\psi},V_{\psi},\mathscr{K}_{\psi}\right)$
be the minimal Stinespring dilation from \prettyref{cor:d1}, i.e.,
$\mathscr{K}_{\psi}=H_{\tilde{K}_{\psi}}=$ the RKHS of $\tilde{K}_{\psi}$,
and 
\[
\tilde{K}\left(\left(A,a\right),\left(B,b\right)\right)=\left\langle a,\psi\left(A^{*}B\right)b\right\rangle _{H}
\]
for all $A,B\in\mathfrak{A}$, and $a,b\in H$.

Then $\varphi\leq\psi$ if and only if there exists a unique positive
selfadjoint operator $T$ in the commutant $\pi_{\psi}\left(\mathfrak{A}\right)'$,
$0\leq T\leq I_{\mathscr{K}_{\psi}}$, such that 
\begin{equation}
\varphi\left(A\right)=V_{\psi}^{*}T^{1/2}\pi_{\psi}\left(A\right)T^{1/2}V_{\psi},\quad A\in\mathfrak{A}.\label{eq:D2}
\end{equation}
\end{cor}

\begin{proof}
Suppose $\varphi\leq\psi$ (i.e., $K_{\varphi}\leq K_{\psi}$). 

Set $K_{\psi}\left(A,B\right)=\psi\left(A^{*}B\right)$ as before,
so it factors into 
\[
K_{\psi}\left(A,B\right)=V_{\psi}\left(A\right)^{*}V_{\psi}\left(B\right).
\]
By \prettyref{thm:c4}, there exists a unique $T$, $0\leq T\leq I_{H_{\tilde{K}_{\psi}}}$,
such that (see \eqref{eq:C4}) 
\begin{equation}
K_{\varphi}\left(A,B\right)=V_{\psi}\left(A\right)^{*}TV_{\psi}\left(B\right).\label{eq:d2}
\end{equation}

Set $A=I$ in \eqref{eq:d2}, then 
\begin{equation}
\varphi\left(B\right)=K_{\varphi}\left(I,B\right)=V_{\psi}\left(I\right)^{*}TV_{\psi}\left(B\right)=V_{\psi}T\pi_{\psi}\left(B\right)V_{\psi},\label{eq:d4}
\end{equation}
where $V_{\psi}:=V_{\psi}\left(I\right)$, and $\pi_{\psi}:\mathfrak{A}\rightarrow\mathcal{L}\left(\mathscr{K}_{\psi}\right)$
is as in \prettyref{cor:d1}. In particular, 
\begin{equation}
V_{\psi}\left(A\right)=\pi_{\psi}\left(A\right)V_{\psi},\quad A\in\mathfrak{A}.\label{eq:d5}
\end{equation}

It remains to show that $T\in\pi_{\psi}\left(\mathfrak{A}\right)'$.
For this, one checks that 
\begin{eqnarray}
\varphi\left(B\right) & = & \varphi\left(\left(B^{*}\right)^{*}I\right)\nonumber \\
 & = & K_{\varphi}\left(B^{*},I\right)\nonumber \\
 & \underset{\left(\ref{eq:d2}\right)}{=} & V_{\psi}\left(B^{*}\right)^{*}TV_{\psi}\left(I\right)\nonumber \\
 & \underset{\left(\ref{eq:d5}\right)}{=} & \left(\pi_{\psi}\left(B^{*}\right)V_{\psi}\right)^{*}TV_{\psi}\nonumber \\
 & = & V_{\psi}^{*}\pi_{\psi}\left(B\right)TV_{\psi}.\label{eq:d6}
\end{eqnarray}
Combining \eqref{eq:d4} and \eqref{eq:d6}, we conclude that $\pi_{\psi}\left(B\right)T=T\pi_{\psi}\left(B\right)$,
for all $B\in\mathfrak{A}$. Thus, $T\in\pi_{\psi}\left(\mathfrak{A}\right)'$
and \eqref{eq:D2} follows. 

Conversely, given \eqref{eq:D2}, it is clear that $\varphi$ is CP.
See e.g., the proof of ``(2)$\Rightarrow$(1)'' in \prettyref{thm:c4}. 
\end{proof}
In view of the correspondence between operator-valued p.d. kernels
and Hilbert space-valued Gaussian processes, we have:
\begin{cor}
Every CP map $\psi:\mathfrak{A}\rightarrow\mathcal{L}\left(H\right)$
admits a direct integral decomposition
\begin{equation}
\psi\left(A^{*}B\right)=\int_{\Omega}\left|W_{\psi}\left(A\right)\left\rangle \right\langle W_{\psi}\left(B\right)\right|d\mathbb{P}\label{eq:d7}
\end{equation}
where $\left\{ W_{\psi}\left(A\right)\right\} _{A\in\mathfrak{A}}$
is a mean-zero, $H$-valued Gaussian process, realized in some probability
space $\left(\Omega,\mathbb{P}\right)$.
\end{cor}

\begin{proof}
Let $\left(V_{\psi},\pi_{\psi},\mathscr{K}_{\psi}\right)$ be the
minimal Stinespring dilation in \prettyref{cor:d1}, i.e., $\mathscr{K}_{\psi}=H_{\tilde{K}_{\psi}}=$
the RKHS of $\tilde{K}_{\psi}$. Then, 
\begin{align*}
\psi\left(A^{*}B\right) & =V_{\psi}^{*}\pi_{\psi}\left(A^{*}B\right)V_{\psi}\\
 & =\left(\pi_{\psi}\left(A\right)V_{\psi}\right)^{*}\left(\pi_{\psi}\left(B\right)V_{\psi}\right)\\
 & =\sum_{i}\left|\left(\pi_{\psi}\left(A\right)V_{\psi}\right)^{*}\varphi_{i}\left\rangle \right\langle \left(\pi_{\psi}\left(B\right)V_{\psi}\right)^{*}\varphi_{i}\right|
\end{align*}
where $\left(\varphi_{i}\right)$ is an ONB in $\mathscr{K}_{\psi}$.
Setting 
\[
W_{\psi}\left(A\right):=\sum_{i}\left(\left(\pi_{\psi}\left(A\right)V_{\psi}\right)^{*}\varphi_{i}\right)Z_{i}
\]
with $\left\{ Z_{i}\right\} $ i.i.d. $N\left(0,1\right)$ as above,
we get 
\[
\left\langle a,\psi\left(A^{*}B\right)b\right\rangle _{H}=\mathbb{E}\left[\left\langle a,W_{\psi}\left(A\right)\right\rangle _{H}\left\langle W_{\psi}\left(B\right),b\right\rangle _{H}\right]
\]
for all $a,b\in H$, and $A,B\in\mathfrak{A}$, which is \eqref{eq:d7}.
\end{proof}
\bibliographystyle{amsalpha}
\bibliography{ref}

\end{document}